\documentclass[reqno,11pt]{amsart}

\usepackage[top=3.0cm,bottom=3cm,left=4cm,right=4cm]{geometry}
\usepackage{amsthm,amsmath,amssymb}
\usepackage{mathrsfs,amsfonts,dsfont,functan,extarrows,mathtools}
\usepackage[colorlinks]{hyperref}

\usepackage{marginnote}
\usepackage{xcolor}

\makeatletter
\newcommand{\rmnum}[1]{\romannumeral #1}
\newcommand{\Rmnum}[1]{\expandafter\@slowromancap\romannumeral #1@}
\makeatother

\newtheorem{theorem}{Theorem}[section]
\newtheorem{proposition}[theorem]{Proposition}

\newtheorem{corollary}[theorem]{Corollary}
\newtheorem{remark}[theorem]{Remark}
\newtheorem{lemma}[theorem]{Lemma}

\numberwithin{equation}{section}
\allowdisplaybreaks

\newcounter{wronumber}

\begin{document}

\title[Classification 4D Gradient Ricci Solitons]{On the classification of four-dimensional gradient Ricci solitons}

\author[F. Yang]{Fei Yang}
\address[Fei Yang]{\newline Corresponding author.\newline School of Mathematics and Physics, China University of Geosciences, Wuhan, 430074, P. R. China}
\email{yangfei810712@163.com}

\author[L. Zhang]{Liangdi Zhang}
\address[Liangdi Zhang]{\newline Center of Mathematical Sciences, Zhejiang University, Hangzhou, 310027, P. R. China}
\email{zhangliangdi@zju.edu.cn}

\thanks{This work is partially supported by Natural Science Foundation of China (No. 11601495) and Science Foundation for The Excellent Young Scholars of Central Universities (No. CUGL170213).}

\date{}


\begin{abstract}
In this paper, we prove some classification results for four-dimensional gradient Ricci solitons. For a four-dimensional gradient shrinking Ricci soliton with $div^4Rm^\pm=0$, we show that it is either Einstein or a finite quotient of $\mathbb{R}^4$, $\mathbb{S}^2\times\mathbb{R}^2$ or $\mathbb{S}^3\times\mathbb{R}$. The same result can be obtained under the condition of $div^4W^\pm=0$. We also present some classification results of four-dimensional complete non-compact gradient expanding Ricci soliton with non-negative Ricci curvature and gradient steady Ricci solitons under certain curvature conditions.

\vspace*{5pt}

\end{abstract}

\maketitle

\tableofcontents

\section{Introduction} 
\label{sec:1}


A complete Riemannian manifold $(M^n,g,f)$ is called a gradient Ricci soliton if there exists a smooth function $f$ on $M^n$ such that the Ricci tensor $Ric$ of the metric $g$ satisfies the equation
\begin{equation}\label{eq:1.1}
Ric+\nabla^2f=\lambda g.
\end{equation}
for some constant $\lambda$. For $\lambda>0$ the Ricci soliton is shrinking, for $\lambda=0$ it is steady and for $\lambda<0$ expanding. By scaling the metric $g$, one customarily assumes
$\lambda\in\{\frac{1}{2}, 0, -\frac{1}{2}\}$.

The classification of gradient Ricci solitons has been a subject of interest for many people in recent years. Z. H. Zhang\cite{zhang17} showed that an $n$-dimensional gradient Ricci shrinking soliton with vanishing Weyl twnsor is a finite quotient of $\mathbb{R}^n$, $\mathbb{S}^{n-1}\times\mathbb{R}$, or $\mathbb{S}^n$ (see also the works of  M. Eminenti-G. La Nave-C. Mantegazza \cite{EMN7}, X. Cao-B. Wang-Z. Zhang \cite{xcao} and P. Peterson-W. Wylie\cite{ref12}). Under the weaker condition of harmonic Weyl tensor, M. Fern\'{a}ndez-L\'{o}pez-E. Garc\'{i}a-R\'{i}o \cite{sref8} and O. Munteanu-N. Sesum.\cite{MS} proved that $n$-dimensional complete gradient shrinking solitons are rigid.

H. D. Cao and Q. Chen \cite{cao Bach} showed that an $n$-dimensional $(n\geq5)$ Bach-flat gradient shrinking Ricci soliton is either Einstein or a finite quotient of $\mathbb{R}^n$ or $\mathbb{R}\times N^{n-1}$, where $N$ is an $(n-1)$-dimensional Einstein manifold. G. Catino, P. Mastrolia and D. D. Monticelli \cite{CMM4} proved that a gradient shrinking Ricci soliton with fourth order divergence-free Weyl tensor (i.e. $div^4W=0$) is rigid.

In the case $n=4$, some stronger results for gradient Ricci shrinking solitons were obtained. A. Naber \cite{naber12} proved that a four-dimensional non-compact shrinking Ricci soliton with bounded nonnegative Riemannian curvature is a finite quotient of $\mathbb{R}^4$, $\mathbb{R}^2\times\mathbb{S}^2$ or $\mathbb{R}\times\mathbb{S}^3$. H. D. Cao and Q. Chen \cite{cao Bach} classified four-dimensional complete non-compact Bach-flat gradient shrinking Ricci solitons. Moreover, they also showed that a compact half-conformally flat (i.e. $W^\pm=0$) gradient shrinking Ricci soliton is isometric to the standard $\mathbb{S}^4$ or $\mathbb{C}P^2$. More generally, J. Y. Wu, P. Wu and W. Wylie \cite{wuwuwylie16} showed that a four-dimensional gradient shrinking Ricci soliton with half harmonic Weyl tensor (i.e. $divW^\pm=0$) is either Einstein or a finite quotient of $\mathbb{R}^4$, $\mathbb{R}^2\times\mathbb{S}^2$ or $\mathbb{R}\times\mathbb{S}^3$. Under the weaker condition of $div^4W^\pm=0$, we will classify four-dimensional gradient shrinking Ricci solitons in Theorem \ref{thm:1.2} of this paper.

Some classification theorems for gradient expanding Ricci solitons have been proved in the recent years by various authors. H. D. Cao et al. \cite{cao Bach s} proved that a complete Bach-flat gradient expanding Ricci soliton with non-negative Ricci curvature is rotationally symmetric. For a three-dimensional gradient expanding Ricci soliton with constant scalar curvature, P. Peterson and W. Wylie \cite{ref12} proved that it is a quotient of $\mathbb{R}^3$, $\mathbb{H}^2\times\mathbb{R}$, and $\mathbb{H}^3$. Moreover, G. Catino, P. Mastrolia and D. D. Monticelli \cite{ref6} showed that a three-dimensional complete gradient expanding Ricci soliton with non-negative Ricci curvature and the scalar curvature $R\in L^1(M^3)$ is isometric to a quotient of the Gaussian soliton $\mathbb{R}^3$.

On gradient steady Ricci solitons, H. D. Cao and Q. Chen \cite{ref3} proved that an $n$-dimensional complete non-compact locally conformally flat gradient steady Ricci soliton is either flat or isometric to the Bryant soliton. Moreover, H. D. Cao et al. \cite{cao Bach s} showed that a Bach-flat gradient steady Ricci soliton with positive Ricci curvature such that the scalar curvature $R$ attains its maximum at some interior point is isometric to the Bryant soliton up to a scaling. Under the conditions of asymptotically cylindrical and positive sectional curvature, S. Brendle \cite{ref2} proved that a steady gradient Ricci soliton  is isometric to the Bryant soliton up to scaling.

In dimension three, the classification of complete gradient steady Ricci solitons is still open. H. D. Cao et al. \cite{cao Bach s} proved that a three-dimensional gradient steady Ricci soliton with divergence-free Bach tensor is either flat or isometric to the Bryant soliton up to a scaling factor. In the paper by S. Brendle \cite{ref1}, it was shown that a three-dimensional complete non-flat and $\kappa$-noncollapsed gradient steady Ricci soliton is isometric to the Bryant soliton up to scaling.

When $n=4$, X. Chen and Y. Wang \cite{chenwang6} showed that a four-dimensional complete gradient steady Ricci soliton with $W^+=0$ must be isometric to the Bryant soliton (up to a scaling) or a manifold which is anti-self-dual and Ricci flat.

In this paper, we focus on the classification of four-dimensional gradient Ricci solitons. The main results of this paper are the following classification theorems for four-dimensional gradient Ricci solitons.

For four-dimensional gradient shrinking Ricci solitons, we have the following classification results.
\begin{theorem}\label{thm:1.1}
Let $(M^4,g,f)$ be a four-dimensional gradient shrinking Ricci soliton. If $div^4Rm^\pm=0$, then $(M^4,g,f)$ is either Einstein or a finite quotient of $\mathbb{R}^4$, $\mathbb{S}^2\times\mathbb{R}^2$ or $\mathbb{S}^3\times\mathbb{R}$.
\end{theorem}
\begin{remark}\label{rmk:1.1}
From the proof of Theorem \ref{thm:1.1}, we know that the condition of $div^4Rm^\pm=0$ can be relaxed to $\int div^4Rm^\pm e^{-f}\geq0$. Moreover, it is clear from the proof of Theorem \ref{thm:1.1} that a four-dimensional gradient shrinking Ricci soliton with $div^3Rm^\pm(\nabla f)=0$ is either Einstein or a finite quotient of $\mathbb{R}^4$, $\mathbb{S}^2\times\mathbb{R}^2$ or $\mathbb{S}^3\times\mathbb{R}$.
\end{remark}

\begin{theorem}\label{thm:1.2}
Let $(M^4,g,f)$ be a four-dimensional gradient shrinking Ricci soliton. If $div^4W^\pm=0$, then $(M^4,g,f)$ is either Einstein or a finite quotient of $\mathbb{R}^4$, $\mathbb{S}^2\times\mathbb{R}^2$ or $\mathbb{S}^3\times\mathbb{R}$.
\end{theorem}

\begin{remark}\label{rmk:1.2}
From the proof of Theorem \ref{thm:1.2}, we know that the condition of $div^4W^\pm=0$ can be relaxed to $\int div^4W^\pm e^{-f}\geq0$. Moreover, it is clear from the proof of Theorem \ref{thm:1.2} that a four-dimensional gradient shrinking Ricci soliton with $div^3W^\pm(\nabla f)=0$ is either Einstein or a finite quotient of $\mathbb{R}^4$, $\mathbb{S}^2\times\mathbb{R}^2$ or $\mathbb{S}^3\times\mathbb{R}$.
\end{remark}

For four-dimensional complete non-compact gradient expanding Ricci solitons with non-negative Ricci curvature, we will prove the following classification theorems.

\begin{theorem}\label{thm:1.3}
Let $(M^4,g,f)$ be a four-dimensional complete non-compact gradient expanding Ricci soliton with non-negative Ricci curvature. If $div^4Rm^\pm=0$, then $(M^4,g,f)$ is a finite quotient of $\mathbb{R}^4$.
\end{theorem}
\begin{remark}\label{rmk:1.3}
From the proof of Theorem \ref{thm:1.3}, we know that the condition of $div^4Rm^\pm=0$  can be relaxed to $\int div^4Rm^\pm e^{-f}\leq0$. Moreover, it is clear from the proof of Theorem \ref{thm:1.3} that a four-dimensional  complete non-compact gradient expanding Ricci soliton with $Ric\geq0$ and $div^3Rm^\pm(\nabla f)=0$ is a finite quotient of $\mathbb{R}^4$.
\end{remark}

\begin{theorem}\label{thm:1.4}
Let $(M^4,g,f)$ be a four-dimensional complete non-compact gradient expanding Ricci soliton with non-negative Ricci curvature. If $div^4W^\pm=0$, then $(M^4,g,f)$ is a finite quotient of $\mathbb{R}^4$.
\end{theorem}

\begin{remark}\label{rmk:1.4}
From the proof of Theorem \ref{thm:1.4}, we know that the condition of $div^4W^\pm=0$ can be relaxed to $\int div^4W^\pm e^{-f}\leq0$.  Moreover, it is clear from the proof of Theorem \ref{thm:1.4} that a four-dimensional  complete non-compact gradient expanding Ricci soliton with $Ric\geq0$ and  $div^3Rm^\pm(\nabla f)=0$ is a finite quotient of $\mathbb{R}^4$.
\end{remark}

For four-dimensional non-trivial complete non-compact steady Ricci solitons, we will show the following classification theorems.
\begin{theorem}\label{thm:1.5}
Let $(M^4,g,f)$ be a non-trivial complete non-compact four-dimensional gradient steady Ricci soliton with $\int|Rm|^2e^{\alpha f}<+\infty$ for some constant $\alpha\in\mathbb{R}$. If $div^3Rm^\pm(\nabla f)=0$, then $(M^4,g,f)$ is a finite quotient of $\mathbb{R}^4$.
\end{theorem}

\begin{theorem}\label{thm:1.6}
Let $(M^4,g,f)$ be a non-trivial complete non-compact four-dimensional gradient steady Ricci soliton with $\int|Rm|^2e^{\alpha f}<+\infty$ for some constant $\alpha\in\mathbb{R}$. If $div^3W^\pm(\nabla f)=0$, then $(M^4,g,f)$ is a finite quotient of $\mathbb{R}^4$.
\end{theorem}

We arrange this paper as follows. In Section \ref{sec: 2}, we fix the notations, recall some basic facts and known results about gradient Ricci solitons that we shall need in the proof of the main theorems. In Section \ref{sec: 3}, we prove some useful formulas of the Riemannian curvature. Before we prove the main results on four-dimensional gradient shrinking Ricci solitons, we prove some integral identities for four-dimensional gradient shrinking Ricci solitons in Section \ref{sec: 4}. In Section \ref{sec: 5}, we prove Theorem \ref{thm:1.1} and Theorem \ref{thm:1.2}. We obtain some integral identities for four-dimensional gradient expanding Ricci solitons with non-negative Ricci curvatures in Section \ref{sec: 6}. In Section \ref{sec: 7}, we finish the proof of Theorem \ref{thm:1.3} and Theorem \ref{thm:1.4}. We show some integral identities for four-dimensional gradient steady Ricci solitons in Section \ref{sec: 8} and prove two classification theorems for four-dimensional gradient steady Ricci solitons (Theorem \ref{thm:1.5} and Theorem \ref{thm:1.6}) in Section \ref{sec: 9}.

\section{Preliminaries} 
\label{sec: 2}

First of all, we recall that on any $n$-dimensional $(n\geq3)$ Riemannian manifold, the Weyl tensor is given by
\begin{eqnarray*}
W_{ijkl}&:=&R_{ijkl}-\frac{1}{n-2}(g_{ik}R_{jl}-g_{il}R_{jk}-g_{jk}R_{il}+g_{jl}R_{ik})\notag\\
&&+\frac{R}{(n-1)(n-2)}(g_{ik}g_{jl}-g_{il}g_{jk}).
\end{eqnarray*}
The Cotton tensor is defined as
$$C_{ijk}:=\nabla_iR_{jk}-\nabla_jR_{ik}-\frac{1}{2(n-1)}(g_{jk}\nabla_iR-g_{ik}\nabla_jR).$$

The relation between the Cotton tensor and the divergence of the Weyl tensor is
\[C_{ijk}=-\frac{n-2}{n-3}\nabla_lW_{ijkl}.\]

For any pair $(ij)$, $1\leq i\neq j\leq4$, denote $(i'j')$ to be the dual of $(ij)$, i.e., the pair such $e_i\wedge e_j\pm e_{i'}\wedge e_{j'}\in\wedge^{\pm}M^4$. In other words, $(iji'j')=\sigma(1234)$ for some even permutation $\sigma\in S_4$, i.e., $$(iji'j')\in\{(1234),(1342),(1423),(2143),(2314),(2431),(3124),(3241),(3412),(4132),(4213),(4321)\}.$$

 For any $(0,4)$-tensor $T$, its (anti-)self-dual part is
\[T^\pm_{ijkl}=\frac{1}{4}(T_{ijkl}\pm T_{ijk'l'}\pm T_{i'j'kl}+T_{i'j'k'l'}).\]

On four-manifolds, the Weyl tensor has sufficiently exotic symmetries.
\begin{proposition}[J. Y. Wu, P. Wu and W. Wylie \cite{wuwuwylie16}]\label{prop:2.0}
Let $(M, g)$ be a four-dimensional Riemannian manifold. Then
\[W_{ijkl}=W_{i'j'k'l'},\]
therefore,
\[W_{ijkl}^\pm=\pm W_{ijk'l'}^\pm=\pm W_{i'j'kl}^\pm=W_{i'j'k'l}^\pm=\frac{1}{2}(W_{ijkl}\pm W_{ijk'l'}).\]
\end{proposition}

It follows that
\[W^\pm_{ijkl}=\frac{1}{2}(W_{ijkl}\pm W_{i'j'kl}).\]
Therefore, we have
\[(divW^\pm)_{ijk}:=\nabla_lW^\pm_{ijkl}=\frac{1}{2}(\nabla_lW_{ijkl}\pm \nabla_lW_{i'j'kl}),\]
\[(div^2W^\pm)_{ik}:=\nabla_j\nabla_lW^\pm_{ijkl}=\frac{1}{2}(\nabla_j\nabla_lW_{ijkl}\pm \nabla_j\nabla_lW_{i'j'kl}),\]
and
\[(div^3W^\pm)_{i}:=\nabla_k\nabla_j\nabla_lW^\pm_{ijkl}=\frac{1}{2}(\nabla_k\nabla_j\nabla_lW_{ijkl}\pm \nabla_k\nabla_j\nabla_lW_{i'j'kl}),\]
\[(div^4W^\pm):=\nabla_i\nabla_k\nabla_j\nabla_lW^\pm_{ijkl}=\frac{1}{2}(\nabla_i\nabla_k\nabla_j\nabla_lW_{ijkl}\pm \nabla_i\nabla_k\nabla_j\nabla_lW_{i'j'kl}).\]

At the regular point of $f$, we denote by $\{e_i\}_{i=1}^4$ a local orthonormal frame of with $e_1=\frac{\nabla f}{|\nabla f|}$. We use $\{\alpha_i\}_{i=1}^4$ to represent eigenvalues of the Ricci tensor with corresponding orthonormal eigenvectors $\{e_i\}_{i=1}^4$.

Next, we recall some basic facts about complete gradient shrinking Ricci solitons.
\begin{proposition}[F. Yang and L. Zhang \cite{yangzhang}]\label{prop:2.1}
Let $(M^n,g)$ be a gradient Ricci soliton with \eqref{eq:1.1}. Then we have the following equations
\begin{equation}\label{eq:2.1}
(div^3Rm)_i=-R_{ijkl}\nabla_kR_{jl},
\end{equation}
\begin{equation}\label{eq:2.2}
div^3Rm(\nabla f)=-\frac{1}{2}|divRm|^2,
\end{equation}
\begin{equation}\label{eq:2.3}
div^3W(\nabla f)=\frac{n-3}{n-2}div^3Rm(\nabla f)+\frac{n-3}{4(n-1)(n-2)}|\nabla R|^2.
\end{equation}
\end{proposition}

\begin{lemma}[H. D. Cao and D. Zhou \cite{caozhou}]\label{lem:2.2}
Let $(M^n,g)$ be a complete gradient shrinking soliton with \eqref{eq:1.1}. Then,

(\rmnum{1}) the potential function f satisfies the estimates
\begin{equation}\label{eq:2.4}
\frac{1}{4}(r(x)-c_1)^2\leq f(x)\leq \frac{1}{4}(r(x)+c_2)^2,
\end{equation}
where $r(x)=d(x_0,x)$ is the distance function from some fixed point $x_0\in M$, $c_1$
and $c_2$ are positive constants depending only on $n$ and the geometry of $g$ on the
unit ball $B(x_0,1)$;

(\rmnum{2}) there exists some constant $C>0$ such that
\begin{equation}\label{eq:2.5}
Vol(B(x_0,s))\leq Cs^n
\end{equation}
for $s>0$ sufficiently large.
\end{lemma}

\begin{lemma}[O. Munteanu and N. Sesum \cite{MS}]\label{lem:2.3}
 For any complete gradient shrinking Ricci soliton with \eqref{eq:1.1}, we have
\begin{equation}\label{eq:2.6}
\int_M|Ric|^2e^{-\rho f}<+\infty
\end{equation}
for any $\rho>0$.
\end{lemma}

\begin{lemma}[J. Y. Wu, P. Wu and W. Wylie \cite{wuwuwylie16}]\label{lem:2.4}
A four-dimensional gradient shrinking Ricci soliton
with $divW^\pm=0$ is either Einstein, or a finite quotient $\mathbb{S}^3\times\mathbb{R}$, $\mathbb{S}^2\times\mathbb{R}^2$ or $\mathbb{R}^4$.
\end{lemma}

\begin{lemma}[H. D. Cao et al. \cite{cao Bach s}]\label{lem:2.5}
Let $(M^n,g_{ij},f)$ $(n\geq3)$ be a complete noncompact gradient expanding soliton
with nonnegative Ricci curvature $Rc\geq0$. Then, there exist some constants $c_1>0$ and $c_2>0$ such that the potential function $f$ satisfies the estimates
\begin{equation}\label{eq:2.7}
\frac{1}{4}(r(x)-c_1)^2-c_2\leq -f(x)\leq\frac{1}{4}(r(x)+2\sqrt{-f(O)})^2,
\end{equation}
where $r(x)$ is the distance function from any fixed base point in $M^n$. In particular, $f$ is a strictly concave exhaustion function achieving its maximum at some interior point $O$, which we take as the base point, and the underlying manifold $M^n$ is diffeomorphic to $\mathbb{R}^n$.
\end{lemma}

\begin{lemma}[P. Petersen and W. Wylie \cite{rigid}]\label{lem:2.6}
The following conditions for a shrinking (or expanding) gradient
soliton $Ric+Hess f=\lambda g$ all imply that the metric is radially flat and has constant
scalar curvature.

(1) The scalar curvature is constant and $sec(E,\nabla f)\geq0$ (or $sec(E,\nabla f)\leq0$).

(2) The scalar curvature is constant and $0\leq Ric\leq g$ (or $\lambda g\leq Ric\leq0$).

(3) The curvature tensor is harmonic.

(4) $Ric\geq0$ (or $Ric\leq0$) and $sec(E,\nabla f)=0$.
\end{lemma}

\begin{lemma}[P. Petersen and W. Wylie \cite{rigid}]\label{lem:2.7}
 A gradient soliton is rigid if and only if it has constant scalar curvature and is radially flat, that is, $sec(E,\nabla f)=0$.
\end{lemma}
\section{Basic Formulas on Curvatures} 
\label{sec: 3}

In this section, we give some formulas of the Riemannian curvature that are needed in the proof of the main results.

\begin{proposition}
On a four-dimensional Riemannian manifold, we have
\begin{equation}\label{eq:3.1}
R_{i'j'ik}=R_{j'ii'k}=R_{ii'j'k}=0,
\end{equation}
and
\begin{equation}\label{eq:3.2}
\nabla_iR_{i'j'kl}=\nabla_{i'}R_{j'ikl}=\nabla_{j'}R_{ii'kl}=0.
\end{equation}
\end{proposition}
\begin{proof}
By direct computations, we have
\[R_{i'1'ik}=R_{432k}+R_{243k}+R_{324k}=0,\]
\[R_{i'2'ik}=R_{341k}+R_{413k}+R_{134k}=0,\]
and
\[R_{i'3'ik}=R_{421k}+R_{142k}+R_{214k}=0,\]
\[R_{i'4'ik}=R_{231k}+R_{312k}+R_{123k}=0.\]
Therefore, we have $R_{i'j'ik}=0$. The same arguments implies that $R_{j'ii'k}=R_{ii'j'k}=0$, i.e. \eqref{eq:3.1} holds.

From the second Bianchi identity and \eqref{eq:3.1}, we have
\[\nabla_iR_{i'j'kl}=\nabla_{k}R_{i'j'il}-\nabla_lR_{i'j'ik}=0.\]

The same arguments implies that $\nabla_{i'}R_{j'ikl}=\nabla_{j'}R_{ii'kl}=0$, i.e. \eqref{eq:3.2} holds.
\end{proof}

As a direct corollary of \eqref{eq:3.2}, we have
\begin{corollary}\label{cor:3.2}
On a four-dimensional Riemannian manifold, we have
\begin{equation}\label{eq:c1}
(divRm^\pm)_{ijk}:=\nabla_lR^\pm_{ijkl}=\frac{1}{4}(\nabla_lR_{ijkl}\pm \nabla_lR_{i'j'kl}),
\end{equation}
\begin{equation}\label{eq:c2}
(div^2Rm^\pm)_{ik}:=\nabla_j\nabla_lR^\pm_{ijkl}=\frac{1}{4}(\nabla_j\nabla_lR_{ijkl}\pm \nabla_j\nabla_lR_{i'j'kl}),
\end{equation}
and
\begin{equation}\label{eq:c3}
(div^3Rm^\pm)_{i}:=\nabla_k\nabla_j\nabla_lR^\pm_{ijkl}=\frac{1}{4}(\nabla_k\nabla_j\nabla_lR_{ijkl}\pm \nabla_k\nabla_j\nabla_lR_{i'j'kl}),
\end{equation}
\begin{equation}\label{eq:c4}
div^4Rm^\pm:=\nabla_i\nabla_k\nabla_j\nabla_lR^\pm_{ijkl}=\frac{1}{4}(\nabla_i\nabla_k\nabla_j\nabla_lR_{ijkl}\pm \nabla_i\nabla_k\nabla_j\nabla_lR_{i'j'kl}).
\end{equation}
\end{corollary}

\begin{proposition}\label{prop:3.3}
On a four-dimensional Riemannian manifold, $divRm^\pm=0$ implies $divW^\pm=0$.
\end{proposition}
\begin{proof}
Tracing $divRm^\pm$, we have
\begin{equation}\label{eq:3.3}
g^{ik}\nabla_lR^\pm_{ijkl}=\frac{1}{4}g^{ik}(\nabla_lR_{ijkl}\pm\nabla_lR_{i'j'kl})=\frac{1}{4}(\nabla_lR_{jl}\pm \nabla_lR_{i'j'il})=\frac{1}{8}\nabla_jR,
\end{equation}
where we used \eqref{eq:c1} in the first equality. Moreover, we used $divRic=\frac{1}{2}\nabla R$ and \eqref{eq:3.1} in the last.

Since $divRm^\pm=0$, \eqref{eq:3.3} implies that $\nabla R=0$.

By direct computation, we have
\begin{eqnarray}\label{eq:3.4}
\nabla_lW^\pm_{ijkl}&=&\frac{1}{2}(\nabla_lW_{ijkl}\pm\nabla_lW_{i'j'kl})\notag\\
&=&-\frac{1}{4}(C_{ijk}\pm C_{i'j'k})\notag\\
&=&-\frac{1}{4}[\nabla_{i}R_{jk}-\nabla_{j}R_{ik}-\frac{1}{6}(g_{jk}\nabla_{i}R-g_{ik}\nabla_{j}R)]\notag\\
&&\mp\frac{1}{4}[(\nabla_{i'}R_{j'k}-\nabla_{j'}R_{i'k}-\frac{1}{6}(g_{j'k}\nabla_{i'}R-g_{i'k}\nabla_{j'}R)]\notag\\
&=&\frac{1}{4}[\nabla_lR_{ijkl}+\frac{1}{6}(g_{jk}\nabla_{i}R-g_{ik}\nabla_{j}R)]\notag\\
&&\pm\frac{1}{4}[\nabla_lR_{i'j'kl}+\frac{1}{6}(g_{j'k}\nabla_{i'}R-g_{i'k}\nabla_{j'}R)]\notag\\
&=&\nabla_lR^\pm_{ijkl}+\frac{1}{24}(g_{jk}\nabla_{i}R-g_{ik}\nabla_{j}R)\pm\frac{1}{24}(g_{j'k}\nabla_{i'}R-g_{i'k}\nabla_{j'}R),
\end{eqnarray}
where we used the second Bianchi identity in the fourth equality and \eqref{eq:c1} in the last.

Note that $divRm^\pm=0$ and $\nabla R=0$, it follows from \eqref{eq:3.4} that $\nabla_lW^\pm_{ijkl}=0$.
\end{proof}

\begin{remark}\label{rem:3.4}
From Lemma \ref{lem:2.4}, we know that a four-dimensional gradient shrinking Ricci soliton with $divRm^\pm=0$ is either Einstein, or a finite quotient $\mathbb{S}^3\times\mathbb{R}$, $\mathbb{S}^2\times\mathbb{R}^2$ or $\mathbb{R}^4$.
\end{remark}

\begin{proposition}\label{prop:3.5}
On a four-dimensional Riemannian manifold, $|divRm^\pm|^2\geq\frac{1}{48}|\nabla R|^2$.
\end{proposition}
\begin{proof}
From \eqref{eq:3.3}, we know that $trdivRm^\pm=\frac{1}{8}\nabla R$. Using \eqref{eq:3.4}, we have
\begin{eqnarray}\label{eq:3.5}
0\leq|\nabla_lW^\pm_{ijkl}|^2&=&|\nabla_lR^\pm_{ijkl}+\frac{1}{24}(g_{jk}\nabla_{i}R-g_{ik}\nabla_{j}R)\pm\frac{1}{24}(g_{j'k}\nabla_{i'}R-g_{i'k}\nabla_{j'}R)|^2\notag\\
&=&|\nabla_lR^\pm_{ijkl}|^2+\frac{1}{576}|g_{jk}\nabla_{i}R-g_{ik}\nabla_{j}R|^2+\frac{1}{576}|g_{j'k}\nabla_{i'}R-g_{i'k}\nabla_{j'}R|^2\notag\\
&&+\frac{1}{12}\nabla_lR^\pm_{ijkl}(g_{jk}\nabla_{i}R-g_{ik}\nabla_{j}R)\pm\frac{1}{12}\nabla_lR^\pm_{ijkl}(g_{j'k}\nabla_{i'}R-g_{i'k}\nabla_{j'}R)\notag\\
&&\pm\frac{1}{288}(g_{jk}\nabla_{i}R-g_{ik}\nabla_{j}R)(g_{j'k}\nabla_{i'}R-g_{i'k}\nabla_{j'}R)\notag\\
&=&|div Rm^\pm|^2+\frac{1}{96}|\nabla R|^2+\frac{1}{96}|\nabla R|^2-\frac{1}{48}|\nabla R|^2-\frac{1}{48}|\nabla R|^2\notag\\
&=&|div Rm^\pm|^2-\frac{1}{48}|\nabla R|^2.
\end{eqnarray}
It follows that $|divRm^\pm|^2\geq\frac{1}{48}|\nabla R|^2$.
\end{proof}

\begin{proposition}\label{prop:3.6}
Let $(M^4,f,g)$ be a four-dimensional gradient Ricci soliton. Then we have
\begin{equation}
\nabla_{j}\nabla_{l}R_{i'j'kl}=-R_{i'j'kl}R_{jl}.
\end{equation}
\end{proposition}

\begin{proof}
The second Bianchi identity implies that
\begin{equation}\label{eq:3.7}
\nabla_j\nabla_{l}R_{i'j'kl}=\nabla_j\nabla_{j'}R_{i'k}-\nabla_j\nabla_{i'}R_{j'k}.
\end{equation}

From \eqref{eq:1.1}, we have
\begin{eqnarray}\label{eq:3.8}
\nabla_jR_{ik}-\nabla_iR_{jk}&=&\nabla_i\nabla_{j}\nabla_{k}f-\nabla_j\nabla_{i}\nabla_{k}f\notag\\
&=&R_{ijkl}\nabla_{l}f.
\end{eqnarray}

By direct computation, we obtain
\begin{eqnarray}\label{eq:3.9}
\nabla_j\nabla_{j'}R_{i'k}&=&\nabla_j(\nabla_kR_{j'i'}+R_{kj'i'l}\nabla_lf)\notag\\
&=&\nabla_j\nabla_kR_{j'i'}+\nabla_jR_{kj'i'l}\nabla_lf+R_{kj'i'l}\nabla_j\nabla_lf\notag\\
&=&\nabla_k\nabla_jR_{j'i'}+R_{jkj'l}R_{li'}+R_{jki'l}R_{j'l}\notag\\
&&+\nabla_jR_{kj'i'l}\nabla_lf+R_{kj'i'l}(\lambda g_{jl}-R_{jl})\notag\\
&=&\nabla_k\nabla_jR_{j'i'}+R_{jkj'l}R_{li'}+R_{jki'l}R_{j'l}-R_{kj'i'l}R_{jl},
\end{eqnarray}
where we used $\eqref{eq:3.8}$ in the first identity, $\eqref{eq:1.1}$ in the third and Proposition \ref{prop:3.1} in the last.

By the same arguments, we have
\begin{eqnarray}\label{eq:3.10}
\nabla_j\nabla_{i'}R_{j'k}&=&\nabla_k\nabla_jR_{i'j'}+R_{jki'l}R_{lj'}+R_{jkj'l}R_{i'l}-R_{ki'j'l}R_{jl}.
\end{eqnarray}

Applying $\eqref{eq:3.9}$ and $\eqref{eq:3.10}$ to $\eqref{eq:3.7}$, we have
\begin{eqnarray*}
\nabla_j\nabla_{l}R_{i'j'kl}&=&\nabla_j\nabla_{j'}R_{i'k}-\nabla_j\nabla_{i'}R_{j'k}\notag\\
&=&(R_{ki'j'l}-R_{kj'i'l})R_{jl}\notag\\
&=&-R_{i'j'kl}R_{jl},
\end{eqnarray*}
where we used the first Bianchi identity.

This completes the proof of Proposition \ref{prop:3.6}.
\end{proof}

\begin{proposition}\label{prop:3.x}
Let $(M^4,f,g)$ be a four-dimensional gradient Ricci soliton. Then we have
\begin{equation}\label{eq:3.x}
R_{i'j'kl}R_{jl}\nabla_if\nabla_kf=-R_{i'j'kl}\nabla_j\nabla_lf\nabla_i f\nabla_k f.
\end{equation}
\end{proposition}

\begin{proof}
By direct computation, we have
\begin{eqnarray}
R_{i'j'kl}R_{jl}\nabla_if\nabla_kf&=&\lambda R_{i'j'kj}\nabla_i f\nabla_k f-R_{i'j'kl}\nabla_j\nabla_lf\nabla_i f\nabla_k f\notag\\
&=&-R_{i'j'kl}\nabla_j\nabla_lf\nabla_i f\nabla_k f,
\end{eqnarray}
where we \eqref{eq:1.1} in the first equality and \eqref{eq:3.1} in the second.
\end{proof}

\begin{proposition}\label{prop:3.7}
Let $(M^4,f,g)$ be a four-dimensional gradient Ricci soliton. Then at every regular point of $f$, we have
\begin{equation}\label{eq:3.11}
R_{i'j'kl}\nabla_iR_{jl}\nabla_kf=2(R_{121l}R_{341l}+R_{131l}R_{421l}+R_{141l}R_{231l})|\nabla f|^2.
\end{equation}
\end{proposition}
\begin{proof}

By direct computation, we have
\begin{eqnarray}\label{eq:3.13}
R_{i'j'kl}\nabla_iR_{jl}\nabla_kf&=&\frac{1}{2}(\nabla_iR_{jl}-\nabla_jR_{il})R_{i'j'kl}\nabla_kf\notag\\
&=&\frac{1}{2}R_{ijpl}\nabla_{p}fR_{i'j'kl}\nabla_kf\notag\\
&=&\frac{1}{2}R_{ij1l}R_{i'j'1l}|\nabla f|^2\notag\\
&=&\frac{1}{2}(R_{121l}R_{341l}+R_{131l}R_{421l}+R_{141l}R_{231l})|\nabla f|^2\notag\\
&&+\frac{1}{2}(R_{211l}R_{431l}+R_{231l}R_{141l}+R_{241l}R_{311l})|\nabla f|^2\notag\\
&&+\frac{1}{2}(R_{311l}R_{241l}+R_{321l}R_{411l}+R_{341l}R_{121l})|\nabla f|^2\notag\\
&&+\frac{1}{2}(R_{411l}R_{321l}+R_{421l}R_{131l}+R_{431l}R_{211l})|\nabla f|^2\notag\\
&=&2(R_{121l}R_{341l}+R_{131l}R_{421l}+R_{141l}R_{231l})|\nabla f|^2,
\end{eqnarray}
where we used \eqref{eq:3.8} in the second equality.
\end{proof}

\begin{proposition}\label{prop:3.8}
Let $(M^4,f,g)$ be a four-dimensional gradient Ricci soliton. Then at every regular point of $f$, we have
\begin{equation}\label{eq:3.14}
|divRm|^2=2(R^2_{121l}+R^2_{341l}+R^2_{131l}+R^2_{421l}+R^2_{141l}+R^2_{231l})|\nabla f|^2.
\end{equation}
\end{proposition}
\begin{proof}
Note that
\begin{eqnarray}\label{eq:3.15}
\nabla_kR_{ijkl}&=&\nabla_iR_{jl}-\nabla_jR_{il}\notag\\
&=&-\nabla_i\nabla_{j}\nabla_{l}f+\nabla_j\nabla_{i}\nabla_{l}f\notag\\
&=&R_{ijkl}\nabla_{k}f,
\end{eqnarray}
where we used the second Bianchi identity in the first equality and \eqref{eq:1.1} in the second equality.

It follows that
\begin{eqnarray}\label{eq:3.16}
|divRm|^2&\equiv&|\nabla_kR_{ijkl}|^2\notag\\
&=&|R_{ijkl}\nabla_kf|^2\notag\\
&=&|R_{ij1l}|^2|\nabla f|^2\notag\\
&=&2(R^2_{121l}+R^2_{341l}+R^2_{131l}+R^2_{421l}+R^2_{141l}+R^2_{231l})|\nabla f|^2.
\end{eqnarray}
\end{proof}

\begin{lemma}\label{lem:3.9}
Let $(M^4,f,g)$ be a four-dimensional gradient Ricci soliton. Then we have
\begin{equation}\label{eq:3.17}
|divRm|^2\pm 2R_{i'j'kl}\nabla_iR_{jl}\nabla_kf=16|divRm^\pm|^2.
\end{equation}
\end{lemma}
\begin{proof} From \eqref{eq:3.15}, we know that $|divRm|=0$ at the critical point of $f$. Moreover, \eqref{eq:c1} and \eqref{eq:3.15} implies that $|divRm^\pm|=0$ at the critical point of $f$. It follows that both side of \eqref{eq:3.17} are zero at the critical point of $f$.

Next, we only consider regular points of $f$.

From Propositions \ref{prop:3.7} and \ref{prop:3.8}, we have
\begin{eqnarray*}
&&|divRm|^2\pm 2R_{i'j'kl}\nabla_iR_{jl}\nabla_kf\notag\\
&=&2(R^2_{121l}+R^2_{341l}+R^2_{131l}+R^2_{421l}+R^2_{141l}+R^2_{231l})|\nabla f|^2\notag\\
&&\pm4(R_{121l}R_{341l}+R_{131l}R_{421l}+R_{141l}R_{231l})|\nabla f|^2\notag\\
&=&2[(R_{121l}\pm R_{341l})^2+(R_{131l}\pm R_{421l})^2+(R_{141l}\pm R_{231l})^2]|\nabla f|^2\notag\\
&=&|R_{ij1l}\pm R_{i'j'1l}|^2|\nabla f|^2\notag\\
&=&|R_{ijkl}\nabla_kf\pm R_{i'j'kl}\nabla_kf|^2\notag\\
&=&|\nabla_kR_{ijkl}\pm \nabla_kR_{i'j'kl}|^2\notag\\
&=&16|divRm^\pm|^2,
\end{eqnarray*}
where we used \eqref{eq:3.15} in the fifth equality and \eqref{eq:c1} in the last.
\end{proof}

\begin{lemma}\label{lem:3.10}
Let $(M^4,f,g)$ be a four-dimensional gradient Ricci soliton. Then we have
\begin{equation}\label{eq:3.18}
div^3W^\pm(\nabla f)=div^3Rm^\pm(\nabla f)+\frac{1}{48}|\nabla R|^2.
\end{equation}
\end{lemma}
\begin{proof} Since $\nabla R=2Ric(\nabla f,\cdot)$, both sides of \eqref{eq:3.18} are zero at the critical point of $f$. In the following, we only consider regular points of $f$.
Note that
\begin{eqnarray*}
\nabla_lW_{i'j'kl}&=&-\frac{1}{2}C_{i'j'k}\notag\\
&=&-\frac{1}{2}[\nabla_{i'}R_{j'k}-\nabla_{j'}R_{i'k}-\frac{1}{6}(g_{j'k}\nabla_{i'}R-g_{i'k}\nabla_{j'}R)]\notag\\
&=&\frac{1}{2}\nabla_lR_{i'j'kl}+\frac{1}{12}(g_{j'k}\nabla_{i'}R-g_{i'k}\nabla_{j'}R),
\end{eqnarray*}
where we used the second Bianchi identity. It follows that
\begin{eqnarray}\label{eq:3.19}
&&\nabla_k\nabla_j\nabla_lW_{i'j'kl}\nabla_if\notag\\
&=&\frac{1}{2}\nabla_k\nabla_j\nabla_lR_{i'j'kl}\nabla_if+\frac{1}{12}(\nabla_{j'}\nabla_j\nabla_{i'}R-\nabla_{i'}\nabla_j\nabla_{j'}R)\notag\\
&=&\frac{1}{2}\nabla_k\nabla_j\nabla_lR_{i'j'kl}\nabla_if+\frac{1}{12}R_{j'i'jl}\nabla_lR\notag\\
&=&\frac{1}{2}\nabla_k\nabla_j\nabla_lR_{i'j'kl}\nabla_if,
\end{eqnarray}
where we used \eqref{eq:3.1}.

By direct computation, we have
\begin{eqnarray*}
&&div^3W^\pm(\nabla f)\notag\\
&=&\frac{1}{2}(\nabla_k\nabla_j\nabla_lW_{ijkl}\nabla_i f\pm \nabla_k\nabla_j\nabla_lW_{i'j'kl}\nabla_i f)\notag\\
&=&\frac{1}{4}\nabla_k\nabla_j\nabla_lR_{ijkl}\nabla_i f+\frac{1}{48}|\nabla R|^2\pm\frac{1}{4}\nabla_k\nabla_j\nabla_lR_{i'j'kl}\nabla_if\notag\\
&=&div^3Rm^\pm(\nabla f)+\frac{1}{48}|\nabla R|^2.
\end{eqnarray*}
where we used \eqref{eq:2.3} and \eqref{eq:3.19} in the second equality. Moreover, we used \eqref{eq:c3} in the last.
\end{proof}
\section{Integral Identities for 4-dimensional Gradient Shrinking Ricci Solitons } 
\label{sec: 4}
In this section, we prove some integral identities for 4-dimensional gradient shrinking Ricci solitons which will be needed in the proof of Theorem \ref{thm:1.1} and Theorem \ref{thm:1.2}.

For the complete non-compact case, let $\phi(t)=1$ on $[0,s]$, $\phi(t)=\frac{2s-t}{s}$ on $(s,2s)$ and $\phi(t)=0$ on $[2s,\infty)$. From \eqref{eq:2.4}, we know that $f$ is of quadratic growth. Therefore, $\phi(f)$ has compact support in $M^4$ for any fixed constant $s>0$.
\begin{proposition}\label{prop:4.1}
Let $(M^4,g,f)$ be a four-dimensional complete non-compact gradient shrinking Ricci soliton, then we have
\begin{equation}\label{eq:4.1}
\int R_{i'j'kl}R_{jl}\nabla_if\nabla_kf\phi(f)e^{-f}=0.
\end{equation}
\end{proposition}
\begin{proof}

Integrating \eqref{eq:3.x}, we have
\begin{eqnarray}\label{4.x1}
&&\int R_{i'j'kl}R_{jl}\nabla_if\nabla_kf(\phi(f)-\phi'(f))e^{-f}\notag\\
&=&-\int R_{i'j'kl}\nabla_j\nabla_lf\nabla_i f\nabla_k f(\phi(f)-\phi'(f))e^{-f}\notag\\
&=&\int\nabla_j(R_{i'j'kl}\nabla_i f\nabla_k f(\phi(f)-\phi'(f))e^{-f})\nabla_lf.
\end{eqnarray}

We claim that $\nabla_j(R_{i'j'kl}\nabla_i f\nabla_k f(\phi(f)-\phi'(f))e^{-f})\nabla_lf\equiv0$ on $M^4$. The arguments can be divided into two cases:

$\bullet$ Case 1: $|\nabla f|^2=0$ on some nonempty open set. In this case, since any gradient shrinking Ricci soliton is analytic in harmonic coordinates, it follows that $|\nabla f|^2=0$ on $M^4$. It is clear that $\nabla_j(R_{i'j'kl}\nabla_i f\nabla_k f(\phi(f)-\phi'(f))e^{-f})\nabla_lf\equiv0$ on $M^4$.

$\bullet$ Case 2: The set $\Theta:=\{x\in M^4|\nabla f(x)\neq0\}$ is dense in $M^4$. Note that $$\nabla_j(R_{i'j'kl}\nabla_i f\nabla_k f(\phi(f)-\phi'(f))e^{-f})\nabla_lf=\nabla_j(R_{1'j'11}|\nabla f|^2(\phi(f)-\phi'(f))e^{-f})|\nabla f|=0$$ on $\Theta$, the continuity implies that $\nabla_j(R_{i'j'kl}\nabla_i f\nabla_k f\phi(f)e^{-f})\nabla_lf\equiv0$ on $M^4$.

This completes the proof of Proposition \ref{prop:4.1}.
\end{proof}

\begin{remark}\label{rmk:4.2}
Note that $\phi(f)-a\phi'(f)$ still has compact support in $M^4$ for any $a>0$ and any fixed constant $s>0$. By the same arguments as in the proof of Proposition \ref{prop:4.1}, we obtain that a four-dimensional complete non-compact gradient shrinking Ricci soliton satisfies
\begin{equation}\label{eq:4.x2}
\int R_{i'j'kl}R_{jl}\nabla_if\nabla_kf(\phi(f)-a\phi'(f))e^{-f}=0
\end{equation}
for any $a>0$.
\end{remark}

\begin{remark}\label{rmk:4.3}
From the proof of Proposition \ref{prop:4.1}, we know that a four-dimensional compact gradient shrinking Ricci soliton satisfies
\begin{equation}\label{eq:4.x3}
\int R_{i'j'kl}R_{jl}\nabla_if\nabla_jfe^{-f}=0.
\end{equation}
\end{remark}

\begin{proposition}\label{prop:4.3}
Let $(M^4,g,f)$ be a four-dimensional complete non-compact gradient shrinking Ricci soliton, then we have
\begin{equation}\label{eq:4.5}
\int R_{i'j'kl}R_{jl}\nabla_i\nabla_kf\phi(f)e^{-f}=-\int R_{i'j'kl}\nabla_iR_{jl}\nabla_kf\phi(f)e^{-f}.
\end{equation}
\end{proposition}
\begin{proof}
By direct computation, we have
\begin{eqnarray*}
&&\int R_{i'j'kl}R_{jl}\nabla_i\nabla_kf\phi(f)e^{-f}\notag\\
&=&-\int \nabla_iR_{i'j'kl}R_{jl}\nabla_kf\phi(f)e^{-f}-\int R_{i'j'kl}\nabla_iR_{jl}\nabla_kf\phi(f)e^{-f}\notag\\
&&+\int R_{i'j'kl}R_{jl}\nabla_if\nabla_kf\phi(f)e^{-f}-\int R_{i'j'kl}R_{jl}\nabla_if\nabla_kf\phi'(f)e^{-f}\notag\\
&=&-\frac{1}{2}\int \nabla_iR_{i'j'kj}\nabla_kf\phi(f)e^{-f}+\int \nabla_iR_{i'j'kl}\nabla_j\nabla_lf\nabla_kf\phi(f)e^{-f}\notag\\
&&-\int R_{i'j'kl}\nabla_iR_{jl}\nabla_kf\phi(f)e^{-f}\notag\\
&=&-\int R_{i'j'kl}\nabla_iR_{jl}\nabla_kf\phi(f)e^{-f},
\end{eqnarray*}
where we used \eqref{eq:4.x2} and \eqref{eq:1.1} in the second equality. Moreover, we used \eqref{eq:3.2} in the last equality.
\end{proof}

\begin{remark}\label{rmk:4.2}
From the proof of Proposition \ref{prop:4.3}, we know that a four-dimensional compact gradient shrinking Ricci soliton satisfies
\begin{equation}\label{eq:4.x4}
\int R_{i'j'kl}R_{jl}\nabla_i\nabla_kfe^{-f}=-\int R_{i'j'kl}\nabla_iR_{jl}\nabla_kfe^{-f}.
\end{equation}
\end{remark}

\begin{lemma}\label{lem:4.6}
Let $(M^4,g,f)$ be a four-dimensional complete non-compact gradient shrinking Ricci soliton, then we have
\begin{equation}\label{eq:4.11}
\int  div^3Rm^\pm(\nabla f)\phi(f)e^{-f}=-2\int|divRm^\pm|^2\phi(f)e^{-f}.
\end{equation}
\end{lemma}
\begin{proof}
By direct computation, we have
\begin{eqnarray*}
&&\int div^3Rm^\pm(\nabla f)\phi(f)e^{-f}\notag\\
&=&\frac{1}{4}\int\nabla_k\nabla_j\nabla_lR_{ijkl}\nabla_i f\phi(f)e^{-f}\pm\frac{1}{4}\int\nabla_k\nabla_j\nabla_lR_{i'j'kl}\nabla_i f\phi(f)e^{-f}\notag\\
&=&-\frac{1}{8}\int|divRm|^2\phi(f)e^{-f}\mp\frac{1}{4}\int\nabla_j\nabla_lR_{i'j'kl}\nabla_k\nabla_i f\phi(f)e^{-f}\notag\\
&&\pm\frac{1}{4}\int\nabla_j\nabla_lR_{i'j'kl}\nabla_i f\nabla_kf(\phi(f)-\phi'(f))e^{-f}\notag\\
&=&-\frac{1}{8}\int|divRm|^2\phi(f)e^{-f}\pm\frac{1}{4}\int R_{i'j'kl}R_{jl}\nabla_i\nabla_k f\phi(f)e^{-f}\notag\\
&&\mp\frac{1}{4}\int R_{i'j'kl}R_{jl}\nabla_i f\nabla_kf(\phi(f)-\phi'(f))e^{-f}\notag\\
&=&-\frac{1}{8}\int|divRm|^2\phi(f)e^{-f}\mp\frac{1}{4}\int R_{i'j'kl}\nabla_iR_{jl}\nabla_k f\phi(f)e^{-f}\notag\\
&=&-2\int|divRm^\pm|^2\phi(f)e^{-f},
\end{eqnarray*}
where we used \eqref{eq:c3} in the first equality, \eqref{eq:2.2} in the second equality and Proposition \ref{prop:3.6} in the third equality. Moreover, we used \eqref{eq:4.x2} and Proposition \ref{prop:4.3} in the fourth equality. We used Lemma \ref{lem:3.9} in the last equality.
\end{proof}

\begin{remark}\label{rmk:4.6}
Note that $\phi(f)-a\phi'(f)$ still has compact support in $M^4$ for any $a>0$ and any fixed constant $s>0$. By the same arguments as in the proof of Lemma \ref{lem:4.6}, we obtain that a four-dimensional complete non-compact gradient shrinking Ricci soliton satisfies
\begin{equation}\label{eq:4.12}
\int  div^3Rm^\pm(\nabla f)(\phi(f)-a\phi'(f))e^{-f}=-2\int|divRm^\pm|^2(\phi(f)-a\phi'(f))e^{-f}
\end{equation}
for any $a>0$.
\end{remark}

\begin{remark}\label{rmk:4.7}
From the proof of Lemma \ref{lem:4.6}, we know that a four-dimensional compact gradient shrinking Ricci soliton satisfies
\begin{equation}\label{eq:4.x5}
\int  div^3Rm^\pm(\nabla f)e^{-f}=-2\int|divRm^\pm|^2e^{-f}.
\end{equation}
\end{remark}

\section{Main Results of Gradient Shrinking Ricci Soltions} 
\label{sec: 5}
In this section, we prove Theorems \ref{thm:1.1} and \ref{thm:1.2}.
\begin{theorem}\label{thm:5.2}
Let $(M^4,g,f)$ be a four-dimensional gradient shrinking Ricci soliton. If $div^4Rm^\pm=0$, then $(M^4,g,f)$ is either Einstein or a finite quotient of $\mathbb{R}^4$, $\mathbb{S}^2\times\mathbb{R}^2$ or $\mathbb{S}^3\times\mathbb{R}$.
\end{theorem}
\begin{proof}
$\bullet$ The compact case:

Integrating by parts, we have
\begin{equation}\label{eq:5.x1}
\int div^4Rm^\pm e^{-f}=\int div^3Rm^\pm(\nabla f)e^{-f}=-2\int |divRm^\pm|^2e^{-f},
\end{equation}
where we used \eqref{eq:4.x5}. Since $div^4Rm^\pm=0$, it follows from \eqref{eq:5.x1} that
\begin{equation}\label{eq:5.x2}
\int |divRm^\pm|^2e^{-f}=0.
\end{equation}

$\bullet$ The complete non-compact case:

$\phi(t)$ is defined in Section \ref{sec: 4}, i.e. $\phi(t)=1$ on $[0,s]$, $\phi(t)=\frac{2s-t}{s}$ on $(s,2s)$ and $\phi(t)=0$ on $[2s,\infty)$. Integrating by parts, we have
\begin{eqnarray}\label{eq:5.1}
\int div^4Rm^\pm\phi(f)e^{-f}&=&\int div^3Rm^\pm(\nabla f)(\phi(f)-\phi'(f))e^{-f}\notag\\
&=&-2\int |divRm^\pm|^2(\phi(f)-\phi'(f))e^{-f},
\end{eqnarray}
where we used \eqref{eq:4.12}.

Note that $div^4Rm^\pm=0$, $\phi(f)-\phi'(f)\geq0$ on $M^4$ and $\phi(f)-\phi'(f)=1$ on the compact set $D(\frac{s}{2}):=\{x\in M^4|f(x)\leq \frac{s}{2}\}$. It follows from \eqref{eq:5.1} that
\begin{equation}
\int_{D(\frac{s}{2})} |divRm^\pm|^2e^{-f}=0.
\end{equation}
Taking $s\rightarrow+\infty$, we have
\begin{equation}\label{eq:5.x3}
\int |divRm^\pm|^2e^{-f}=0.
\end{equation}

It follows from \eqref{eq:5.x2} and \eqref{eq:5.x3} that $divRm^\pm=0$ $a.$ $e.$ on $M^4$. Since any gradient shrinking Ricci soliton is analytic in harmonic coordinates, we have $divRm^\pm\equiv0$. From Remark \ref{rem:3.4}, we know that $(M^4,g,f)$ is either Einstein or a finite quotient of $\mathbb{R}^4$, $\mathbb{S}^2\times\mathbb{R}^2$ or $\mathbb{S}^3\times\mathbb{R}$.
\end{proof}

\begin{theorem}\label{thm:5.4}
Let $(M^4,g,f)$ be a four-dimensional gradient shrinking Ricci soliton. If $div^4W^\pm=0$. Then $(M^4,g,f)$ is either Einstein or a finite quotient of $\mathbb{R}^4$, $\mathbb{S}^2\times\mathbb{R}^2$ or $\mathbb{S}^3\times\mathbb{R}$.
\end{theorem}
\begin{proof}
From the proof of Theorem \ref{thm:5.2}, we only need to show that $\int|divRm^\pm|^2e^{-f}=0$.

$\bullet$ The compact case:

Integrating by parts, we have
\begin{eqnarray}
\int div^4W^\pm e^{-f}&=&\int div^3W^\pm(\nabla f)e^{-f}\notag\\
&=&\int div^3Rm^\pm(\nabla f)e^{-f}+\frac{1}{48}\int|\nabla R|^2e^{-f}\notag\\
&=&-2\int|divRm^\pm|^2e^{-f}+\frac{1}{48}\int|\nabla R|^2e^{-f},
\end{eqnarray}
where we used Lemma \ref{lem:3.10} in the second equality and \eqref{eq:4.x5} in the last.

Since $ div^4W^\pm=0$, it follows that
\begin{equation}\label{eq:5.x4}
\int|divRm^\pm|^2e^{-f}=\frac{1}{96}\int|\nabla R|^2e^{-f}
\end{equation}

Applying Proposition \ref{prop:3.5} to \eqref{eq:5.x4}, we have
\begin{equation}\label{eq:5.7}
\frac{1}{48}\int|\nabla R|^2e^{-f}\leq\frac{1}{96}\int|\nabla R|^2e^{-f},
\end{equation}
i. e.
\begin{equation}\label{eq:5.x8}
\int|\nabla R|^2e^{-f}=0.
\end{equation}

From \eqref{eq:5.x4} and \eqref{eq:5.x8}, we know that
\begin{equation}\label{eq:5.x5}
\int|divRm^\pm|^2e^{-f}=0.
\end{equation}

$\bullet$ The complete non-compact case:

$\phi(t)$ is defined in Section \ref{sec: 4}, i.e. $\phi(t)=1$ on $[0,s]$, $\phi(t)=\frac{2s-t}{s}$ on $(s,2s)$ and $\phi(t)=0$ on $[2s,\infty)$. Integrating by parts, we have
\begin{eqnarray}\label{eq:5.5}
&&\int div^4W^\pm\phi(f)e^{-f}\notag\\
&=&\int div^3W^\pm(\nabla f)(\phi(f)-\phi'(f))e^{-f}\notag\\
&=&\int div^3Rm^\pm(\nabla f)(\phi(f)-\phi'(f))e^{-f}+\frac{1}{48}\int|\nabla R|^2(\phi(f)-\phi'(f))e^{-f}\notag\\
&=&-2\int|divRm^\pm|^2(\phi(f)-\phi'(f))e^{-f}+\frac{1}{48}\int|\nabla R|^2(\phi(f)-\phi'(f))e^{-f},
\end{eqnarray}
where we used Lemma \ref{lem:3.10} in the second equality and \eqref{eq:4.12} in the last.

Since $ div^4W^\pm=0$, it follows that
\begin{equation}\label{eq:5.6}
\int|divRm^\pm|^2(\phi(f)-\phi'(f))e^{-f}=\frac{1}{96}\int|\nabla R|^2(\phi(f)-\phi'(f))e^{-f},
\end{equation}

Note that $\phi(f)-\phi'(f)\geq0$ on $M^4$. Applying Proposition \ref{prop:3.5} to \eqref{eq:5.6} that
\begin{equation}\label{eq:5.7}
\frac{1}{48}\int|\nabla R|^2(\phi(f)-\phi'(f))e^{-f}\leq\frac{1}{96}\int|\nabla R|^2(\phi(f)-\phi'(f))e^{-f},
\end{equation}
i. e.
\begin{equation}\label{eq:5.8}
\int|\nabla R|^2(\phi(f)-\phi'(f))e^{-f}=0.
\end{equation}

From \eqref{eq:5.6} and \eqref{eq:5.8}, we know that
\begin{equation}\label{eq:5.9}
\int|divRm^\pm|^2(\phi(f)-\phi'(f))e^{-f}=0.
\end{equation}

Note that $\phi(f)-\phi'(f)=1$ on the compact set $D(\frac{s}{2}):=\{x\in M^4|f(x)\leq \frac{s}{2}\}$. It follows from \eqref{eq:5.9} that
\begin{equation}
\int_{D(\frac{s}{2})}|divRm^\pm|^2=0.
\end{equation}
Taking $s\rightarrow+\infty$, we obtain
\begin{equation}\label{eq:5.10}
\int|divRm^\pm|^2e^{-f}=0.
\end{equation}

This completes the proof of Theorem \ref{thm:5.4}.
\end{proof}

\section{Integral Identities for 4-dimensional Gradient Expanding Ricci Solitons} 
\label{sec: 6}
In this section, we prove some integral identities for 4-dimensional gradient expanding Ricci solitons which will be needed in the proof of Theorem \ref{thm:1.3} to Theorem \ref{thm:1.4}.

$\phi(t)$ is defined in Section \ref{sec: 4}, i.e. $\phi(t)=1$ on $[0,s]$, $\phi(t)=\frac{2s-t}{s}$ on $(s,2s)$ and $\phi(t)=0$ on $[2s,\infty)$. Lemma \ref{lem:2.5} implies that $-f$ is of quadratic growth, it follows that $\phi(-f)=\frac{s+f}{s}$ has compact support in $M$ for any fixed constant $s>0$.
\begin{proposition}\label{prop:6.1}
Let $(M^4,g,f)$ be a four-dimensional complete non-compact gradient expanding Ricci soliton with non-negative Ricci curvature, then we have
\begin{equation}\label{eq:6.1}
\int R_{i'j'kl}R_{jl}\nabla_if\nabla_kf\phi(-f)e^{f}=0.
\end{equation}
\end{proposition}
\begin{proof}
Integrating \eqref{eq:3.x}, we have
\begin{eqnarray*}
&&\int R_{i'j'kl}R_{jl}\nabla_if\nabla_kf\phi(-f)e^{f}\notag\\
&=&-\int R_{i'j'kl}\nabla_j\nabla_lf\nabla_i f\nabla_k f\phi(-f)e^{f}\notag\\
&=&\int\nabla_j(R_{i'j'kl}\nabla_i f\nabla_k f\phi(-f)e^{f})\nabla_lf.
\end{eqnarray*}

We claim that $\nabla_j(R_{i'j'kl}\nabla_i f\nabla_k f\phi(-f)e^{f})\nabla_lf\equiv0$ on $M^4$. The arguments can be divided into two cases:

$\bullet$ Case 1: $|\nabla f|^2=0$ on some nonempty open set. In this case, since any gradient expanding Ricci soliton is analytic in harmonic coordinates, it follows that $|\nabla f|^2=0$ on $M^4$. It is clear that $\nabla_j(R_{i'j'kl}\nabla_i f\nabla_k f\phi(-f)e^{f})\nabla_lf\equiv0$ on $M^4$.

$\bullet$ Case 2: The set $\Theta:=\{x\in M^4|\nabla f(x)\neq0\}$ is dense in $M^4$. Note that $$\nabla_j(R_{i'j'kl}\nabla_i f\nabla_k f\phi(-f)e^{f})\nabla_lf=\nabla_j(R_{1'j'11}|\nabla f|^2\phi(-f)e^{f})|\nabla f|=0$$ on $\Theta$, the continuity implies that $\nabla_j(R_{i'j'kl}\nabla_i f\nabla_k f\phi(-f)e^{f})\nabla_lf\equiv0$ on $M^4$.

This completes the proof of Proposition \ref{prop:6.1}.
\end{proof}

\begin{remark}\label{rmk:6.1}
Note that $\phi(-f)-a\phi'(-f)$ still has compact support in $M^4$ for any $a>0$ and any fixed constant $s>0$. By the same arguments as in the proof of Proposition \ref{prop:6.1}, we know that on a four-dimensional complete non-compact gradient expanding Ricci soliton with non-negative Ricci curvature,
\begin{equation}\label{eq:6.x1}
\int R_{i'j'kl}R_{jl}\nabla_if\nabla_jf(\phi(-f)-a\phi'(-f))e^{f}=0.
\end{equation}
for any $a>0$.
\end{remark}

\begin{proposition}\label{prop:6.2}
Let $(M^4,g,f)$ be a four-dimensional complete non-compact gradient expanding Ricci soliton with non-negative Ricci curvature, then we have
\begin{equation}\label{eq:6.5}
\int R_{i'j'kl}R_{jl}\nabla_i\nabla_kf\phi(-f)e^{f}=-\int R_{i'j'kl}\nabla_iR_{jl}\nabla_kf\phi(-f)e^{f}.
\end{equation}
\end{proposition}
\begin{proof}
By direct computation, we have
\begin{eqnarray*}
&&\int R_{i'j'kl}R_{jl}\nabla_i\nabla_kf\phi(-f)e^{f}\notag\\
&=&-\int \nabla_iR_{i'j'kl}R_{jl}\nabla_kf\phi(-f)e^{f}-\int R_{i'j'kl}\nabla_iR_{jl}\nabla_kf\phi(-f)e^{f}\notag\\
&&-\int R_{i'j'kl}R_{jl}\nabla_if\nabla_kf(\phi(-f)-\phi'(-f))e^{f}\notag\\
&=&\frac{1}{2}\int \nabla_iR_{i'j'kj}\nabla_kf\phi(-f)e^{f}+\int \nabla_iR_{i'j'kl}\nabla_j\nabla_lf\nabla_kf\phi(-f)e^{f}\notag\\
&&-\int R_{i'j'kl}\nabla_iR_{jl}\nabla_kf\phi(-f)e^{f}\notag\\
&=&-\int R_{i'j'kl}\nabla_iR_{jl}\nabla_kf\phi(-f)e^{f},
\end{eqnarray*}
where we used \eqref{eq:6.x1} and \eqref{eq:1.1} in the second equality. Moreover, we used \eqref{eq:3.2} in the last equality.
\end{proof}

\begin{lemma}\label{lem:6.3}
Let $(M^4,g,f)$ be a four-dimensional complete non-compact gradient expanding Ricci soliton with non-negative Ricci curvature, then we have
\begin{equation}\label{eq:6.5}
\int  div^3Rm^\pm(\nabla f)\phi(-f)e^{f}=-2\int|divRm^\pm|^2\phi(-f)e^{f}.
\end{equation}
\end{lemma}
\begin{proof}
By direct computation, we have
\begin{eqnarray*}
&&\int div^3Rm^\pm(\nabla f)\phi(-f)e^{f}\notag\\
&=&\frac{1}{4}\int\nabla_k\nabla_j\nabla_lR_{ijkl}\nabla_i f\phi(-f)e^{f}\pm\frac{1}{4}\int\nabla_k\nabla_j\nabla_lR_{i'j'kl}\nabla_i f\phi(-f)e^{f}\notag\\
&=&-\frac{1}{8}\int|divRm|^2\phi(-f)e^{f}\mp\frac{1}{4}\int\nabla_j\nabla_lR_{i'j'kl}\nabla_k\nabla_i f\phi(-f)e^{f}\notag\\
&&\mp\frac{1}{4}\int\nabla_j\nabla_lR_{i'j'kl}\nabla_i f\nabla_kf(\phi(-f)-\phi'(-f))e^{f}\notag\\
&=&-\frac{1}{8}\int|divRm|^2\phi(-f)e^{f}\pm\frac{1}{4}\int R_{i'j'kl}R_{jl}\nabla_i\nabla_k f\phi(-f)e^{f}\notag\\
&&\pm\frac{1}{4}\int R_{i'j'kl}R_{jl}\nabla_i f\nabla_kf(\phi(-f)-\phi'(-f))e^{f}\notag\\
&=&-\frac{1}{8}\int|divRm|^2\phi(-f)e^{f}\mp\frac{1}{4}\int R_{i'j'kl}\nabla_iR_{jl}\nabla_k f\phi(-f)e^{f}\notag\\
&=&-2\int|divRm^\pm|^2\phi(-f)e^{f},
\end{eqnarray*}
where we used \eqref{eq:c3} in the first equality, \eqref{eq:2.2} in the second equality and Proposition \ref{prop:3.6} in the third equality. Moreover, we used \eqref{eq:6.x1} and Proposition \ref{prop:6.2} in the fourth equality. We used Lemma \ref{lem:3.9} in the last equality.
\end{proof}

\begin{remark}\label{rmk:6.2}
Note that $\phi(-f)-a\phi'(-f)$ still has compact support in $M^4$ for any $a>0$ and any fixed constant $s>0$. By the same arguments as in the proof of Lemma \ref{lem:6.3}, we know that on a four-dimensional complete non-compact gradient expanding Ricci soliton with
\begin{equation}\label{eq:6.x2}
\int  div^3Rm^\pm(\nabla f)(\phi(-f)-a\phi'(-f))e^{f}=-2\int|divRm^\pm|^2(\phi(-f)-a\phi'(-f)e^{f}.
\end{equation}
for any $a>0$.
\end{remark}
\section{Main Results of Gradient Expanding Ricci Soltions} 
\label{sec: 7}
In this section, we finish the proof of Theorem \ref{thm:1.5} and Theorem \ref{thm:1.6}.

$\phi(t)$ is defined in Section \ref{sec: 4}, i.e. $\phi(t)=1$ on $[0,s]$, $\phi(t)=\frac{2s-t}{s}$ on $(s,2s)$, $\phi=0$ on $[2s,\infty)$.
\begin{theorem}\label{thm:7.2}
Let $(M^4,g,f)$ be a four-dimensional complete non-compact gradient expanding Ricci soliton with non-negative Ricci curvature. If $div^4Rm^\pm=0$, then $(M^4,g,f)$ is a finite quotient of the Gaussian expanding soliton $\mathbb{R}^4$.
\end{theorem}
\begin{proof}
Integrating by parts, we have
\begin{eqnarray}\label{eq:7.1}
&&\int div^4Rm^\pm\phi(-f)e^{f}\notag\\
&=&-\int div^3Rm^\pm(\nabla f)(\phi(-f)-\phi'(-f))e^{f}\notag\\
&=&2\int |divRm^\pm|^2(\phi(-f)-\phi'(-f))e^{f},
\end{eqnarray}
where we used \eqref{eq:6.x2}.

Note that $div^4Rm^\pm=0$, $\phi(-f)-\phi'(-f)\geq0$ on $M^4$ and $\phi(-f)-\phi'(-f)=1$ on the compact set $E(\frac{s}{2}):=\{x\in M^4|-f(x)\leq \frac{s}{2}\}$. It follows from \eqref{eq:7.1} that
\begin{equation}\label{eq:7.x1}
\int_{E(\frac{s}{2})} |divRm^\pm|^2e^{f}=0.
\end{equation}

Taking $s\rightarrow+\infty$ in \eqref{eq:7.x1}, we know that
\begin{equation}
\int|divRm^\pm|^2e^{-f}=0.
\end{equation}
It follows that $divRm^\pm=0$ $a.$ $e.$ on $M^4$. Since any gradient expanding Ricci soliton is analytic in harmonic coordinates, we have $divRm^\pm\equiv0$ on $M^4$.  It follows \eqref{eq:3.3} that $\nabla R\equiv0$ on $M^4$, i.e. $R$ is a constant on $M^4$.

Since $Ric\geq0$, $R$ is a non-negative constant. It follows from the fact of $\Delta_fR=-R-2|Ric|^2$ that $|Ric|=0$ on $M^4$, i.e. $M^4$ has vanishing Ricci curvature.

Hence, Condition (2) in Lemma \ref{lem:2.6} holds. It follows that $(M^n,g,f)$ is radially flat and has constant scalar curvature. By Lemma \ref{lem:2.7}, we have $(M^n,g,f)$ is rigid. Since $Ric\geq0$ on $M^4$, $(M^4,g,f)$ is a finite quotient of the Gaussian expanding soliton $\mathbb{R}^4$.
\end{proof}

\begin{theorem}\label{thm:7.4}
Let $(M^4,g,f)$ be a four-dimensional complete non-compact gradient expanding Ricci soliton with non-negative Ricci curvature. If $div^4W^\pm=0$, then $(M^4,g,f)$ is a finite quotient of the Gaussian expanding soliton $\mathbb{R}^4$.
\end{theorem}
\begin{proof}
From the proof of  Theorem \ref{thm:7.2}, we only need to show that $R$ is a constant on $M^4$.

Integrating by parts, we have
\begin{eqnarray}\label{eq:7.5}
&&\int div^4W^\pm\phi(-f)e^{f}\notag\\
&=&-\int div^3W^\pm(\nabla f)(\phi(-f)-\phi'(-f))e^{f}\notag\\
&=&-\int div^3Rm^\pm(\nabla f)(\phi(-f)-\phi'(-f))e^{f}-\frac{1}{48}\int|\nabla R|^2(\phi(-f)-\phi'(-f))e^{f}\notag\\
&=&2\int|divRm^\pm|^2(\phi(-f)-\phi'(-f))e^{f}-\frac{1}{48}\int|\nabla R|^2(\phi(-f)-\phi'(-f))e^{f},
\end{eqnarray}
where we used Lemma \ref{lem:3.10} in the second equality and \eqref{eq:6.x2} in the last.

Since $ div^4W^\pm=0$, it follows that
\begin{equation}\label{eq:7.6}
\int|divRm^\pm|^2(\phi(-f)-\phi'(-f))e^{f}=\frac{1}{96}\int|\nabla R|^2(\phi(-f)-\phi'(-f))e^{f}.
\end{equation}

Note that $\phi(-f)-\phi'(-f)\geq0$ on $M^4$. Applying Proposition \ref{prop:3.5} to \eqref{eq:7.6} that
\begin{equation}\label{eq:7.7}
\frac{1}{48}\int|\nabla R|^2(\phi(-f)-\phi'(-f))e^{f}\leq\frac{1}{96}\int|\nabla R|^2(\phi(-f)-\phi'(-f))e^{f},
\end{equation}
i. e.
\begin{equation}\label{eq:7.8}
\int|\nabla R|^2(\phi(f)-\phi'(f))e^{f}=0.
\end{equation}

Note that $\phi(-f)-\phi'(-f)=1$ on the compact set $E(\frac{s}{2}):=\{x\in M^4|-f(x)\leq \frac{s}{2}\}$. It follows from \eqref{eq:7.8} that
\begin{equation}\label{eq:7.9}
\int_{E(\frac{s}{2})}|\nabla R|^2e^{f}=0.
\end{equation}
Taking $s\rightarrow+\infty$, we obtain
\begin{equation}\label{eq:7.10}
\int|\nabla R|^2e^{f}=0.
\end{equation}

It follows that $\int|\nabla R|^2e^{f}=0$, i.e. $\nabla R=0$ $a.$ $e.$ on $M^4$. Since any gradient expanding Ricci soliton is analytic in harmonic coordinates, we have $\nabla R\equiv0$ on $M^4$, i.e. $R$ is a constant on $M^4$.

This completes the proof of  Theorem \ref{thm:7.4}.
\end{proof}

\section{Integral Identities for 4-dimensional Gradient Steady Ricci Solitons} 
\label{sec: 8}
In this section, we prove some integral identities for 4-dimensional gradient steady Ricci solitons which will be needed in the proof of Theorem \ref{thm:1.5} and Theorem \ref{thm:1.6}.

In the following, let $B_r$ be a geodesic ball with radius $r$ and $\nu$ be the outward unit normal vector field to $B_r$.
\begin{proposition}\label{prop:8.1}
Let $(M^4,g,f)$ be a four-dimensional complete non-compact gradient steady Ricci soliton, then for every $\alpha\in\mathbb{R}$ we have
\begin{equation}\label{eq:8.1}
\int_{B_r} R_{i'j'kl}R_{jl}\nabla_if\nabla_kfe^{\alpha f}=0.
\end{equation}
\end{proposition}
\begin{proof}
Integrating \eqref{eq:3.x}, we have
\begin{eqnarray*}
&&\int_{B_r} R_{i'j'kl}R_{jl}\nabla_if\nabla_kfe^{\alpha f}\notag\\
&=&-\int_{B_r} R_{i'j'kl}\nabla_j\nabla_lf\nabla_i f\nabla_k fe^{\alpha f}\notag\\
&=&-\int_{\partial B_r} R_{i'j'kl}\nabla_lf\nabla_i f\nabla_k f\nu_je^{\alpha f}+\int_{B_r}\nabla_j(R_{i'j'kl}\nabla_i f\nabla_k fe^{\alpha f})\nabla_lf\notag\\
&=&\int_{B_r}\nabla_j(R_{i'j'kl}\nabla_i f\nabla_k fe^{\alpha f})\nabla_lf.
\end{eqnarray*}

We claim that $\nabla_j(R_{i'j'kl}\nabla_i f\nabla_k fe^{\alpha f})\nabla_lf\equiv0$ on $M^4$. The arguments can be divided into two cases:

$\bullet$ Case 1: $|\nabla f|^2=0$ on some nonempty open set. In this case, since any gradient shrinking Ricci soliton is analytic in harmonic coordinates, it follows that $|\nabla f|^2=0$ on $M^4$. It is clear that $\nabla_j(R_{i'j'kl}\nabla_i f\nabla_k fe^{\alpha f})\nabla_lf\equiv0$ on $M^4$.

$\bullet$ Case 2: The set $\Theta:=\{x\in M^4|\nabla f(x)\neq0\}$ is dense in $M^4$. Note that $$\nabla_j(R_{i'j'kl}\nabla_i f\nabla_k fe^{\alpha f})\nabla_lf=\nabla_j(R_{1'j'11}|\nabla f|^2e^{\alpha f})|\nabla f|=0$$ on $\Theta$, the continuity implies that $\nabla_j(R_{i'j'kl}\nabla_i f\nabla_k fe^{\alpha f})\nabla_lf\equiv0$ on $M^4$.

This completes the proof of Proposition \ref{prop:8.1}.
\end{proof}

\begin{proposition}\label{prop:8.2}
Let $(M^4,g,f)$ be a four-dimensional complete non-compact gradient steady Ricci soliton, then for every $\alpha\in\mathbb{R}$ we have
\begin{equation}\label{eq:8.5}
\int_{B_r} R_{i'j'kl}R_{jl}\nabla_i\nabla_kfe^{\alpha f}=\int_{\partial B_r}R_{i'j'kl}R_{jl}\nabla_kf\nu_ie^{\alpha f}-\int_{B_r} R_{i'j'kl}\nabla_iR_{jl}\nabla_kfe^{\alpha f}.
\end{equation}
\end{proposition}
\begin{proof}
By direct computation, we have
\begin{eqnarray*}
&&\int_{B_r} R_{i'j'kl}R_{jl}\nabla_i\nabla_kfe^{\alpha f}\notag\\
&=&\int_{\partial B_r} R_{i'j'kl}R_{jl}\nabla_kf\nu_ie^{\alpha f}-\int_{B_r} \nabla_iR_{i'j'kl}R_{jl}\nabla_kfe^{\alpha f}\notag\\
&&-\int_{B_r} R_{i'j'kl}\nabla_iR_{jl}\nabla_kfe^{\alpha f}-\alpha\int_{B_r} R_{i'j'kl}R_{jl}\nabla_if\nabla_kfe^{\alpha f}\notag\\
&=&\int_{\partial B_r} R_{i'j'kl}R_{jl}\nabla_kf\nu_ie^{\alpha f}-\int_{B_r} R_{i'j'kl}\nabla_iR_{jl}\nabla_kfe^{\alpha f},
\end{eqnarray*}
where we used \eqref{eq:3.2} and Proposition \ref{prop:8.1}.
\end{proof}

\begin{lemma}\label{lem:8.3}
Let $(M^4,g,f)$ be a complete non-compact four-dimensional gradient steady Ricci soliton with $\int|Rm|^2e^{\alpha f}<+\infty$ for some constant $\alpha\in\mathbb{R}$, then we have
\begin{equation}\label{eq:8.6}
\int div^3Rm^\pm(\nabla f)e^{\alpha f}\leq-2\int|divRm^\pm|^2e^{\alpha f}
\end{equation}
with
\begin{equation}\label{eq:8.7}
\int|divRm^\pm|^2e^{\alpha f}<+\infty.
\end{equation}
\end{lemma}
\begin{proof}
It follows from \eqref{eq:c1} that
\[\int|divRm^\pm|^2e^{\alpha f}\leq\int|divRm|^2e^{\alpha f}<+\infty.\]

By direct computation, we have
\begin{eqnarray}\label{eq:8.x}
&&\int_{B_r} div^3Rm^\pm(\nabla f)e^{\alpha f}\notag\\
&=&\frac{1}{4}\int_{B_r}\nabla_k\nabla_j\nabla_lR_{ijkl}\nabla_i fe^{\alpha f}\pm\frac{1}{4}\int_{B_r}\nabla_k\nabla_j\nabla_lR_{i'j'kl}\nabla_i fe^{\alpha f}\notag\\
&=&-\frac{1}{8}\int_{B_r}|divRm|^2e^{\alpha f}\pm\frac{1}{4}\int_{\partial B_r}\nabla_j\nabla_lR_{i'j'kl}\nabla_i f\nu_ke^{\alpha f}\notag\\
&&\mp\frac{1}{4}\int_{B_r}\nabla_j\nabla_lR_{i'j'kl}\nabla_k\nabla_i fe^{\alpha f}\mp\frac{\alpha}{4}\int_{B_r}\nabla_j\nabla_lR_{i'j'kl}\nabla_i f\nabla_kfe^{\alpha f}\notag\\
&=&-\frac{1}{8}\int_{B_r}|divRm|^2e^{\alpha f}\mp\frac{1}{4}\int_{\partial B_r}R_{i'j'kl}R_{jl}\nabla_i f\nu_ke^{\alpha f}\notag\\
&&\pm\frac{1}{4}\int_{B_r} R_{i'j'kl}R_{jl}\nabla_k\nabla_i fe^{\alpha f}\pm\frac{\alpha}{4}\int_{B_r}R_{i'j'kl}R_{jl}\nabla_i f\nabla_kfe^{\alpha f}\notag\\
&=&\mp\frac{1}{4}\int_{\partial B_r}R_{i'j'kl}R_{jl}\nabla_i f\nu_ke^{\alpha f}\pm\frac{1}{4}\int_{\partial B_r}R_{i'j'kl}R_{jl}\nabla_k f\nu_ie^{\alpha f}\notag\\
&&-\frac{1}{8}\int_{B_r}|divRm|^2e^{\alpha f}\mp\frac{1}{4}\int_{B_r} R_{i'j'kl}\nabla_iR_{jl}\nabla_k fe^{\alpha f}\notag\\
&=&\mp\frac{1}{4}\int_{\partial B_r}R_{i'j'kl}R_{jl}\nabla_i f\nu_ke^{\alpha f}\pm\frac{1}{4}\int_{\partial B_r}R_{i'j'kl}R_{jl}\nabla_k f\nu_ie^{\alpha f}-2\int_{B_r}|divRm^\pm|^2e^{\alpha f}\notag\\
&\leq&\frac{1}{2}\int_{\partial B_r}|Rm||Ric||\nabla f|e^{\alpha f}-2\int_{B_r}|divRm^\pm|^2e^{\alpha f},
\end{eqnarray}
where we used \eqref{eq:c3} in the first equality, \eqref{eq:2.2} in the second equality and Proposition \ref{prop:3.6} in the third. Moreover, we used Proposition \ref{prop:8.1} and Proposition \ref{prop:8.2} in the fourth equality. We used Lemma \ref{lem:3.9} in the last equality. Therefore, we obtained \eqref{eq:8.5}.

Since $R\geq0$ (see B. L. Chen \cite{cbl}) and $R+|\nabla f|^2=Const.$, $|\nabla f|$ is bounded. Therefore,
\begin{equation}\label{eq:8.9}
\int |Rm||Ric||\nabla f|e^{\alpha f}\leq c\int |Rm|^2e^{\alpha f}<+\infty.
\end{equation}
It follows that
\begin{equation}\label{eq:8.8}
\lim_{r\rightarrow+\infty}\int_{\partial B_r}|Rm||Ric||\nabla f|e^{\alpha f}=0.
\end{equation}

By taking $r\rightarrow+\infty$ in \eqref{eq:8.x} and using \eqref{eq:8.8}, we obtain \eqref{eq:8.6}.
\end{proof}
\section{Main Results of Gradient Steady Ricci Solitons} 
\label{sec: 9}
In this section, we prove Theorem \ref{thm:1.5} and \ref{thm:1.6}.
\begin{theorem}\label{thm:9.1}
Let $(M^4,g,f)$ be a non-trivial complete non-compact four-dimensional gradient steady Ricci soliton with $\int|Rm|^2e^{\alpha f}<+\infty$ for some constant $\alpha\in\mathbb{R}$. If $div^3Rm^\pm(\nabla f)=0$, then $(M^4,g,f)$ is a finite quotient of $\mathbb{R}^4$.
\end{theorem}
\begin{proof}
It follows from \eqref{eq:8.6} that
\[0\leq\int|divRm^\pm|^2e^{\alpha f}\leq-\int div^3Rm^\pm(\nabla f)e^{\alpha f}=0,\]
i.e.
\[\int|divRm^\pm|^2e^{\alpha f}=0.\]

It follows that $divRm^\pm=0$ $a.$ $e.$ on $M^4$. Since any gradient steady Ricci soliton is analytic in harmonic coordinates, we have $divRm^\pm\equiv0$. From \eqref{eq:3.3}, we know that $\nabla R\equiv0$ on $M^4$, i.e. $R$ is constant on $M^4$.

Since $0=\Delta_fR=-2|Ric|^2$, $(M^n,g,f)$ has vanishing Ricci curvature. By the second Bianchi identity, we have
\begin{equation}
\nabla_lR_{ijkl}=\nabla_jR_{ik}-\nabla_iR_{jk}=0,
\end{equation}
i.e. $divRm=0$ on $M^n$, which implies $M^4$ is radially flat and has constant scalar curvature. By Lemma \ref{lem:2.7}, we have $(M^4,g,f)$ is rigid. Note that $(M^4,g,f)$ is non-trivial, we conclude that $(M^4,g,f)$ is a finite quotient of $\mathbb{R}^4$.
\end{proof}

\begin{theorem}\label{thm:9.2}
Let $(M^4,g,f)$ be non-trivial complete non-compact a four-dimensional gradient steady Ricci soliton with $\int|Rm|^2e^{\alpha f}<+\infty$ for some constant $\alpha\in\mathbb{R}$. If $div^3W^\pm(\nabla f)=0$, then $(M^4,g,f)$ is a finite quotient of $\mathbb{R}^4$.
\end{theorem}
\begin{proof}
From the proof of Theorem \ref{thm:9.1}, we only need to show that $R$ is a constant on $M^4$.

Integrating \eqref{eq:4.1}, we have
\begin{eqnarray}\label{eq:9.2}
0&=&\int div^3W^\pm(\nabla f)e^{\alpha f}\notag\\
&=&\int div^3Rm^\pm(\nabla f)e^{\alpha f}+\frac{1}{48}\int|\nabla R|^2e^{\alpha f}\notag\\
&=&-2\int|divRm^\pm|^2e^{\alpha f}+\frac{1}{48}\int|\nabla R|^2e^{\alpha f},
\end{eqnarray}
where we used Lemma \ref{lem:3.10} in the second equality and Lemma \ref{lem:8.3} in the last.

Applying Proposition \ref{prop:3.5} to \eqref{eq:9.2}, we obtain that
\[\frac{1}{48}\int|\nabla R|^2e^{\alpha f}\leq\int|divRm^\pm|^2e^{\alpha f}=\frac{1}{96}\int|\nabla R|^2e^{\alpha f},\]
i.e.
\[\int|\nabla R|^2e^{\alpha f}=0.\]
It follows that $\nabla R=0$ $a.$ $e.$ on $M^4$. Since any gradient steady Ricci soliton is analytic in harmonic coordinates, we have $\nabla R\equiv0$, i.e. $R$ is a constant on $M^4$.

This completes the proof of Theorem \ref{thm:9.2}.
\end{proof}
\label{sec:acknowledgements}





\end{document}